\documentclass[onecolumn,noversion]{cdcarticle}
\def\MODE{1}

\setlrmargins{1.2in}{1.2in}
\usepackage{empheq}
\usepackage{graphicx,amsmath,amsthm,amssymb,url,mathtools}

\usepackage{hyperref}
\if\MODE1
\hypersetup{colorlinks=true,linkcolor=blue,urlcolor=black, citecolor=blue, breaklinks}
\fi
\if\MODE2
\hypersetup{hidelinks}
\fi
\if\MODE3
\hypersetup{hidelinks}
\fi

\usepackage[english]{babel} 
\usepackage[T1]{fontenc}
\usepackage{lmodern}
\usepackage{lipsum}
\usepackage{url}

\usepackage[title]{appendix}
\usepackage{enumerate}
\usepackage{thmtools}
\usepackage{thm-restate}
\usepackage{color}
\usepackage{commands}

\newcommand*\widefbox[1]{\fbox{\hspace{1em}#1\hspace{1em}}}
\numberwithin{equation}{section}

\begin{document}

\title{Finite-Data Performance Guarantees for the Output-Feedback Control of an\\ Unknown System}

\if\MODE1\author{Ross Boczar \and Nikolai Matni \and Benjamin Recht}
\note{Department of Electrical Engineering and Computer Sciences\\ University of California, Berkeley\\ \today}
\else\author{Ross Boczar\footnotemark[1] \and Nikolai Matni\footnotemark[1] \and Benjamin Recht\footnotemark[1]}\fi

\maketitle

\if\MODE1\else
\footnotetext[1]{R.~Boczar, N.~Matni, and B.~Recht are with the University of California, Berkeley, CA~94720, USA. \texttt{\{boczar,brecht,nmatni\}@berkeley.edu}}
\fi

\begin{abstract}
As the systems we control become more complex, first-principle modeling becomes either impossible or intractable, motivating the use of machine learning techniques for the control of systems with continuous action spaces.  As impressive as the empirical success of these methods have been, strong theoretical guarantees of performance, safety, or robustness are few and far between.  This paper takes a step towards such providing such guarantees by establishing finite-data performance guarantees for the robust output-feedback control of an unknown FIR SISO system.  In particular, we introduce the ``Coarse-ID control'' pipeline, which is composed of a system identification step followed by a robust controller synthesis procedure, and analyze its end-to-end performance, providing quantitative bounds on the performance degradation suffered due to model uncertainty as a function of the number of experiments run to identify the system.  We conclude with numerical examples demonstrating the effectiveness of our method.
\end{abstract}

\section{Introduction}
\label{sec:intro}

There have been many recent results (see for example \cite{ctsRL1,ctsRL2,ctrlML1,ctrlML2,ctrlML3,dean2017sample} and the references within) that apply state-of-the-art machine learning techniques to the control of systems with continuous action spaces. As the systems we control become ever more complex, be it in their dynamics, their scale, or their interaction with the environment, moving to a data-driven approach will be inevitable: in these settings, first-principle modeling becomes either impossible or intractable.  However, as promising and exciting as recent empirical demonstrations of these techniques have been, they have, for the most part, lacked the rigorous stability, safety and robustness guarantees that the controls community has always prided itself in providing. Indeed, such guarantees are not only desirable, but \emph{necessary} when such techniques are being proposed for the control of safety critical systems or infrastructures.

This paper can be seen as a step towards providing such guarantees, albeit in a simplified setting, wherein we establish rigorous baselines of robustness and performance when controlling a single-input-single-output (SISO) system with an unknown transfer function. To do so, we combine contemporary approaches to system identification and robust control into what we term the ``Coarse-ID control'' pipeline.  In particular, we leverage the results developed in \cite{tu17} to provide finite-sample guarantees on optimally (in a certain sense) estimating a stable single-input single-output linear time-invariant (SISO LTI) system, using input-output data pairs.\footnote{We note that there have been recent results in the system identification literature (for example \cite{sysID1,sysID2}) that also seek to provide non-asymptotic guarantees of model estimation quality.}  Such finite-data guarantees are not only in stark contrast to classical system identification results, which typically only provide asymptotic guarantees of model fidelity (see~\cite{ljungBook} for an overview), but also necessary for the principled integration of these techniques with robust control, as they allow us to quantify the amount of uncertainty that our controller must contend with.  We then formulate a robust control problem using the recently developed \emph{system-level synthesis} (SLS) procedure~\cite{SysLevelSyn1}, which exploits a novel parameterization of stabilizing controllers for LTI systems that allows us to quantify performance degradation in terms of the amount of uncertainty affecting the system \cite{SysLevelRobust}.  Again this is in contrast to classical methods from robust control~\cite{ZDGbook} that are only able to provide robust stability guarantees for a prescribed amount of uncertainty.

\paragraph{Main contribution} A feature of ``Coarse-ID control,'' as described above, is that we can analyze the end-to-end performance of this pipeline in a non-asymptotic setting. Specifically, we show that the difference in cost between the optimal cost for the true system (an FIR SISO system of length $r$) and the realized cost induced by instead solving a robust SLS procedure for the approximate system is $O\left(\sqrt{\frac{\sigma^2 r}{m}}\right)$. Here, we assume that the approximate system was estimated using the ``optimal'' coarse-grained system identification procedure described in Tu et al.~\cite{tu17}, with $\sigma^2$ the measurement noise variance and $m$ the number of experiments conducted in order to construct an estimate of the system.  Finally, this paper should be viewed as a step towards generalizing the results in~\cite{dean2017sample}, which provides finite-data end-to-end performance guarantees for the classical LQR optimal control problem, to the output-feedback setting.

\paragraph{Paper organization} In Section~\ref{sec:prelim} we fix notation and quickly outline the structure used by common robust control problems. Section~\ref{sec:sls_robust} then gives an overview of the system-level synthesis framework and how it can be used to solve these problems. Finally, in Section~\ref{sec:optimality} we combine this framework with recent work on coarse-grained identification to provide quantitative bounds on how the performance of a robust controller synthesized using the SLS framework degrades when the plant to be controlled is only approximately identified. We conclude in Section~\ref{sec:experiments} with computational examples.

\section{Preliminaries}
\label{sec:prelim}
\if\MODE1

\paragraph{Notation}
We use boldface to denote frequency domain signals and transfer functions. The $i$-th standard basis vector is given by $e_i$. A discrete-time dynamical system
\begin{align*}
x_{k+1} = &\; Ax_k + B_k u_k\\
y_k = &\; Cx_k + D_k
 u_k\end{align*}
can be represented compactly by $\statespace{A}{B}{C}{D}$ or the tuple $(A,B,C,D)$ (with $(A,B,C)$ implying $D=0$). The set of stable real-rational proper transfer matrices is denoted $\RHinf$. Unless otherwise noted, $\norm{\cdot}$ represents the $\hinf$-norm (the induced $\ell_2 \to \ell_2$ norm) for elements in $\RHinf$ (this reduces to the spectral norm for constant matrices).

\subsection{The standard robust control problem}
We first introduce a standard form for generic robust and optimal control problems, and then show how simple disturbance attenuation and reference tracking problems can be cast into this standard form. We work with discrete-time LTI systems, but unless stated otherwise, all results extend naturally to the continuous-time setting.  A system in standard form can be described by the following equations:
\begin{align}\label{eq:std_form}
\tf z &= \tf P_{11} \tf w + \tf P_{12} \tf u \notag\\
\tf y &= \tf P_{21} \tf w + \tf P_{22} \tf u \\
\tf u &= \tf K\tf y,\notag
\end{align}
where $\tf z$ is the regulated output (e.g., deviations of the system state from a desired set-point), $\tf y$ is the measured output available to the controller $\tf K$ to compute the control action $\tf u = \tf K \tf y$, and $\tf w$ is the exogenous disturbance.  We further assume that the full plant $\tf P$ admits a joint realization\footnote{We assume throughout that $\tf P_{22}$ is strictly proper---it follows that $\tf K (I-\tf P_{22}\tf K)^{-1} \in \RHinf$ is a necesssary and sufficient condition for internal stability of the closed loop system shown in Figure \ref{fig:standard} \cite{ZDGbook}.}, i.e. 
\begin{equation}
\tf P = \begin{bmatrix}  \tf P_{11} & \tf P_{12} \\ \tf P_{21} & \tf P_{22} \end{bmatrix} = \statespace{A}{\begin{matrix} B_1 & B_2 \end{matrix}}{\begin{matrix} C_1 \\ C_2 \end{matrix}}{\begin{matrix} D_{11} & D_{12} \\ D_{21} & 0\end{matrix}}
\label{eq:joint-realization}
\end{equation}
where $\tf P_{ij} = C_i(zI-A)^{-1}B_j + D_{ij}$.

The standard optimal control problem of minimizing the gain from exogenous disturbance $\tf w$ to regulated output $\tf z$, subject to internal stability of the closed loop system shown in Figure \ref{fig:standard}, can then be posed as 
\begin{align}\label{eq:opt_ctrl}
\underset{\tf K}{\minimize} &\; \norm{\tf P_{11} + \tf P_{12} \tf K (I-\tf P_{22}\tf K)^{-1} \tf P_{21}} \\
\st &\;\; \tf K (I-\tf P_{22}\tf K)^{-1} \in \RHinf. \notag
\end{align}
\begin{figure}[t]
\centering
\includegraphics[width=.6\linewidth]{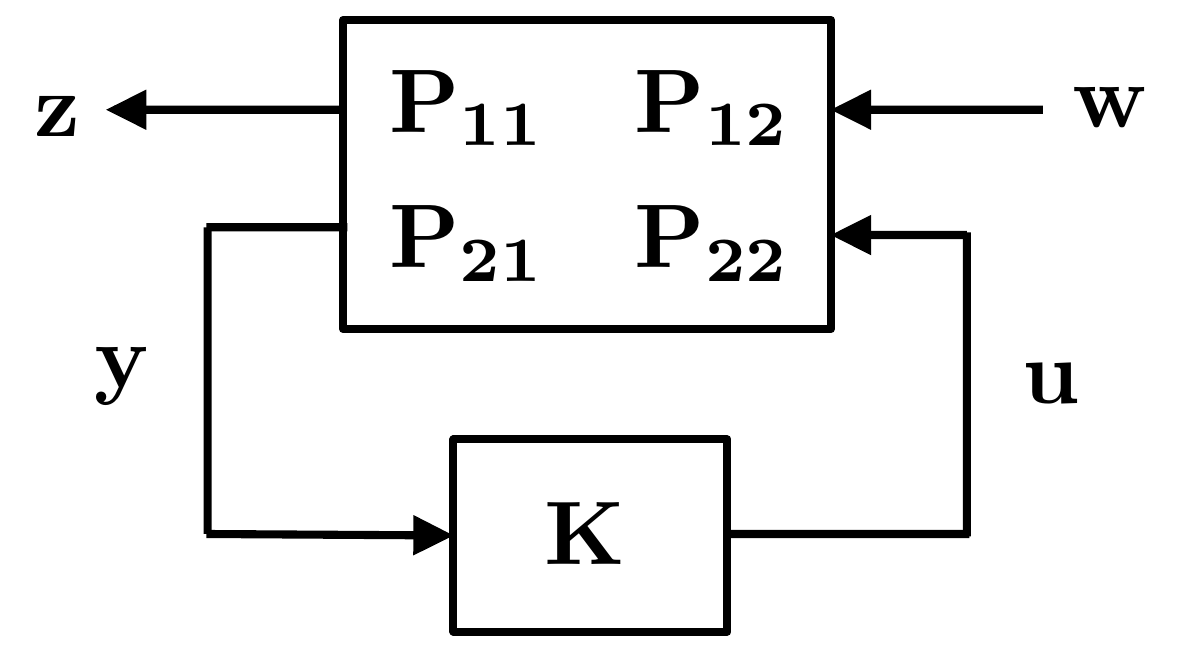}
\caption{The standard optimal control problem \eqref{eq:opt_ctrl} for the plant $\tf P$ as defined in equations \eqref{eq:std_form} and \eqref{eq:joint-realization}.}
\label{fig:standard}
\end{figure}
\paragraph{Disturbance rejection}
Consider the feedback system shown in Figure \ref{fig:dist_att}, wherein a controller $\tf K$ is in feedback with a SISO plant $\tf G$, with input disturbance $\tf d$ and measurement noise $\tf n$.  We can then define the disturbances and outputs as
\begin{align*}
\tf w = &\; \begin{bmatrix}\tf d \\ \tf n\end{bmatrix} \\
\tf z = &\;\begin{bmatrix} \tf v \\ \rho \tf u\end{bmatrix} \\
\tf y = &\;\tf v + \tf n\:,
\end{align*}
respectively, where $\rho > 0$.  Furthermore, let the plant $\tf G$ have a state-space realization $(A,B,C)$.  We then have that
\begin{align*}
\tf z = &\;\begin{bmatrix} \tf G & 0 \\ 0 & 0 \end{bmatrix} \tf w + \begin{bmatrix} \tf G \\ \rho \end{bmatrix}\tf u \\
  := &\;\tf P_{11} \tf w + \tf P_{12} \tf u\\
\tf y = &\;\begin{bmatrix} \tf G & 1 \end{bmatrix} \tf w + \tf G \tf u\\
 := &\; \tf P_{21} \tf w + \tf P_{22} \tf u,
\end{align*}
from which it follows that the generalized plant $\tf P$ admits the joint realization
\begin{equation*}
\tf P = \statespace{A}{\begin{bmatrix}B & 0\end{bmatrix} \, B}{\begin{matrix}\begin{bmatrix}C \\ 0 \end{bmatrix} \\ C \end{matrix}}{\begin{matrix}\begin{bmatrix} 0 & 0 \\ 0 & 0 \end{bmatrix} & \begin{bmatrix} 0 \\ \rho \end{bmatrix} \\ \begin{bmatrix} 0 & 1 \end{bmatrix} & 0 \end{matrix}}\:.
\end{equation*}
\begin{figure}[t]
\centering
\includegraphics[width=.6\linewidth]{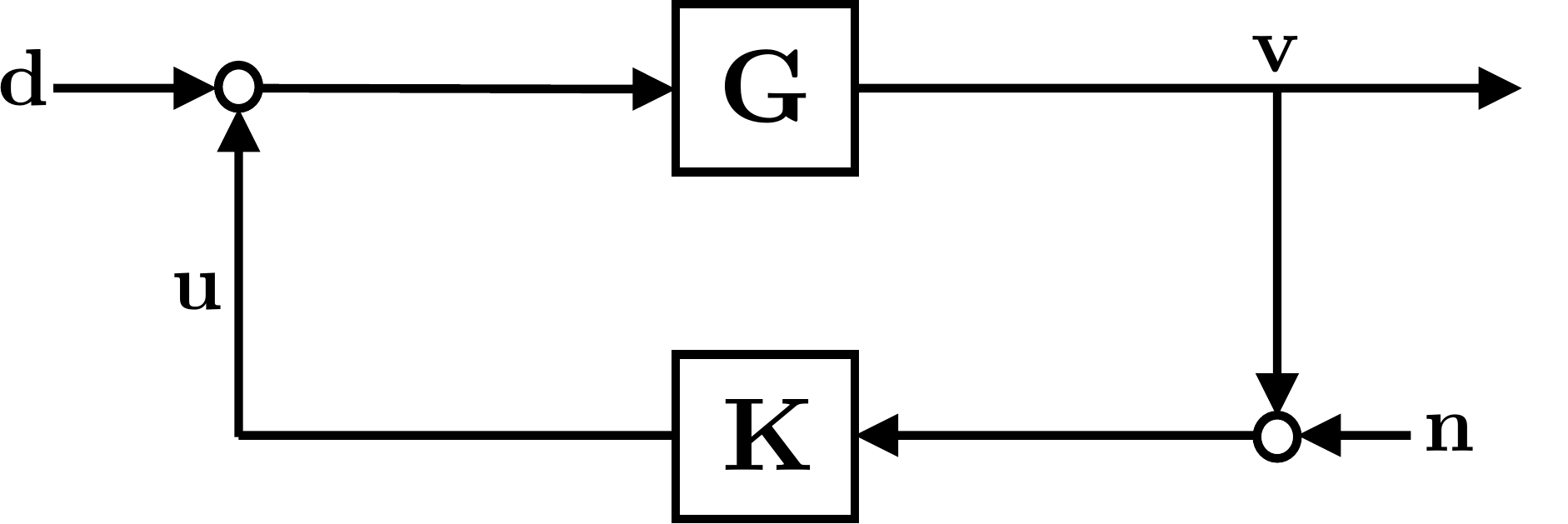}
\caption{The disturbance rejection problem for a SISO plant $\tf G$ with input disturbance $\tf d$ and measurement noise $\tf n$.}
\label{fig:dist_att}
\end{figure}
\paragraph{Reference tracking}
Now, consider the feedback system shown in Figure \ref{fig:ref_track}, wherein a controller $\tf K$ is in feedback with a SISO plant $\tf G$, with input disturbance $\tf d$ and reference signal $\tf r$.  We can then define the disturbances and outputs as
\begin{align*}
\tf w &= \begin{bmatrix}\tf d \\ \tf r\end{bmatrix}\\
\tf z  & = \begin{bmatrix} \tf e \\ \rho \tf u\end{bmatrix}\\
\tf y & = \tf e\:, \end{align*}
respectively, where $\rho > 0$.  Furthermore, let the plant $\tf G$ have a state-space realization $(A,B,C)$.  We then have that
\begin{align*}
\tf z =&\; \begin{bmatrix} \tf G & -1 \\ 0 & 0 \end{bmatrix} \tf w + \begin{bmatrix} \tf G \\ \rho \end{bmatrix} \tf u \\
 := &\;\tf P_{11} \tf w + \tf P_{12} \tf u\\
\tf y =&\; \begin{bmatrix} \tf G & -1 \end{bmatrix} \tf w + \tf G \tf u\\
 :=&\; \tf P_{21} \tf w + \tf P_{22} \tf u\:,
\end{align*}
from which it follows that the full plant $\tf P$ admits the joint realization
\begin{equation*}
\tf P = \statespace{A}{\begin{bmatrix}B & 0\end{bmatrix} \, B}{\begin{matrix}\begin{bmatrix}C \\ 0 \end{bmatrix} \\ C \end{matrix}}{\begin{matrix}\begin{bmatrix} 0 & -1 \\ 0 & 0 \end{bmatrix} & \begin{bmatrix} 0 \\ \rho \end{bmatrix} \\ \begin{bmatrix} 0 & -1 \end{bmatrix} & 0 \end{matrix}}.
\end{equation*}

\begin{figure}[t]
\centering
\includegraphics[width=.6\linewidth]{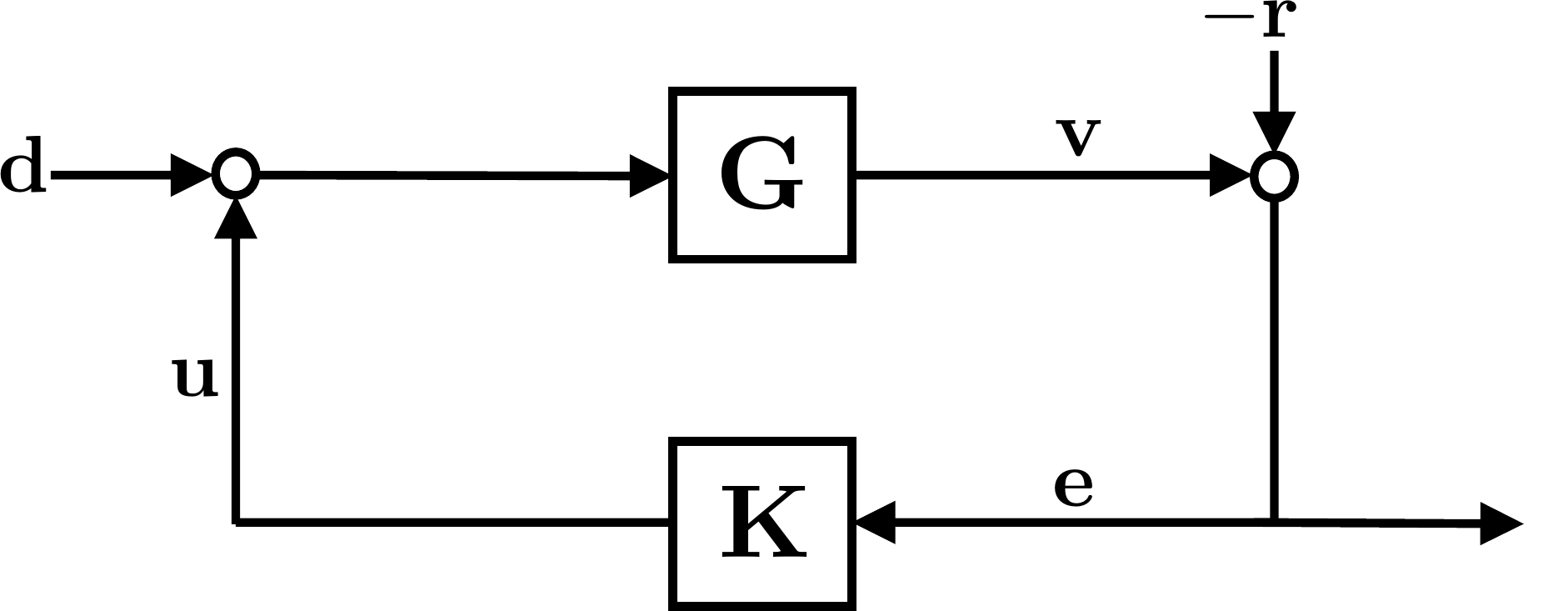}
\caption{The reference tracking problem for a SISO plant $\tf G$ with disturbance $\tf d$ and reference $\tf r$.}
\label{fig:ref_track}
\end{figure}

\paragraph{Specialization to FIR plant G}
Suppose that $\tf G$ is strictly proper and has a finite impulse response (FIR) of order $r$, i.e., that $\tf G = \sum_{t = 1}^{r-1} g_tz^{-t}$ for a collection of real scalars $\{g_t\}_{t=1}^{r-1}$.  Defining $g = [g_1, \, g_2, \, \dots,\, g_{r-1}]^\tp$, the plant $\tf G$ admits the state-space realization $(Z,e_1,g^\tp)$ where $Z$ is the right-shift operator (i.e., a matrix with ones one the sub-diagonal and zeros elsewhere).  Given the examples presented thus far, going forward we assume that
\begin{equation}
C_1 = \begin{bmatrix} g  & 0 \end{bmatrix}^\tp, \ C_2 = g^\tp,
\label{eq:FIR_Cs}
\end{equation}
as well as the standard assumption that $D_{12}^\tp C_1 = 0$.  Additionally, given that we are considering SISO systems, we can without loss of generality (by suitably rescaling $B_2$) assume that 
\begin{equation}
D_{12} = \begin{bmatrix} 0 & 1 \end{bmatrix}^\tp.
\label{eq:FIR_D12}
\end{equation}

\subsection{Coarse-grained identification}
\label{ssec:coarse-id}
As our aim is to provide end-to-end guarantees for robust control problems, we must first have a scheme to acquire an approximate plant model $\tf{\hat G}$. Toward that end, \emph{Coarse-grained identification}, as defined in Tu et al.~\cite{tu17}, describes the following procedure:
\begin{enumerate}[(i)]
\item carefully choose a series of $m$ inputs $\{\tf u_i\}$, where $\tf u_i \in \mathcal U$, and collect noisy outputs $\{\tf y_i\}$ where $\tf y_i = \tf G \tf u_i + \xi_i$ with $\xi_i\stackrel{\text{i.i.d.}}{\sim}\mathcal{N}(0,\sigma^2I)$
\item form a least-squares estimate $\tf{\hat G}$ of the impulse response of $\tf G$ using $\{\tf u_i, \:\tf y_i\}$.
\end{enumerate}
We refer to each such pair $(\tf u_i, \tf y_i)$ as an experiment.  

In~\cite{tu17}, upper and lower bounds are shown on the resulting $\hinf$ error between $\tf G$ and $\tf{\hat G}$ for different sets $\mathcal{U}$. This built on the work of Goldenschluger~\cite{golden1998}, who derived estimation rates for $\ell_\infty$-constrained inputs. We make slight modifications to the results in \cite{tu17} to instead provide $\ell_2$ error bounds on the impulse response coefficients, as these are more natural for our problem. One concern is that we will need $\ell_2$ bounds on the impulse response error, and these bounds are in term of the $\hinf$-norm. However, while they can conservatively be plugged in verbatim (as $\|\{g_k\}\|_2 \leq \|\tf G\|$), we will instead modify their proofs slightly to fit our application.

\else
\input{prelim-abridged}
\fi

\section{System-Level Synthesis}
\label{sec:sls_robust}
\if\MODE1
The System-Level Synthesis (SLS) framework, proposed by Wang et. al~\cite{SysLevelSyn1}, provides a parameterization of stabilizing controllers that achieve specified responses between disturbances and outputs. We briefly review here the SLS framework,  and later show in Section \ref{sec:robust-sls} how it can be modified to solve a robust optimal control problem subject to bounded uncertainty on the FIR coefficients $g$. 

For an LTI system with dynamics described by \eqref{eq:joint-realization}, we define a system response $\{\tf R, \tf M, \tf N, \tf L\}$ to be the maps satisfying
\begin{equation}
\begin{bmatrix} \tf x \\ \tf u \end{bmatrix} = \Thetamat \begin{bmatrix} \tf{\delta}_x \\ \tf{\delta}_y \end{bmatrix}. \label{eq:desired}
\end{equation}
where $\tf{\delta}_x := B_1 \tf w$ is the process noise, and $\tf{\delta}_y := D_{21} \tf w$ is the measurement noise.

We call a system response $\tf \Theta = \{\tf R, \tf M, \tf N, \tf L\}$ \emph{stable and achievable} with respect to a plant $\tf P$ if there exists an internally stabilizing controller $\tf K$ such that the control rule $\tf u = \tf K \tf y$ leads to closed loop behavior consistent with \eqref{eq:desired}. It was shown in \cite{SysLevelSyn1} that the parameterization of all stable and achievable system responses $\{\tf R, \tf M, \tf N, \tf L\}$ is defined by the following affine space:
\begin{subequations} \label{eq:achievable}
\begin{align}
\begin{bmatrix} zI - A & -B_2 \end{bmatrix}
\Thetamat &= 
\begin{bmatrix} I & 0 \end{bmatrix} \label{eq:main_1} \\
\Thetamat
\begin{bmatrix} zI - A \\ -C_2 \end{bmatrix} &= 
\begin{bmatrix} I \\ 0 \end{bmatrix} \label{eq:main_2} \\
\tf R, \tf M, \tf N \in \frac{1}{z} \mathcal{RH}_{\infty}, & \quad \tf L \in \mathcal{RH}_{\infty}. \label{eq:main_stable}
\end{align}
\end{subequations}
We call equations~\eqref{eq:main_1} - \eqref{eq:main_stable} the \emph{SLS constraints}.
The parameterization of all internally stabilizing controllers is given by the following theorem.
\begin{theorem}[Theorem 2 in \cite{SysLevelSyn1}] \label{thm:param} 
Suppose that a system response $\{\tf R, \tf M, \tf N, \tf L\}$ satisfies the SLS constraints \eqref{eq:main_1} - \eqref{eq:main_stable}. Then, $\tf K = \tf L - \tf M \tf R^{-1}\tf N$ is an internally stabilizing controller for the plant \eqref{eq:joint-realization} that yields the desired system response \eqref{eq:desired}.  Furthermore, the solutions of \eqref{eq:main_1} - \eqref{eq:main_stable} with the implementation $\tf K = \tf L - \tf M \tf R^{-1}\tf N$ parameterize all internally stabilizing controllers for the plant \eqref{eq:joint-realization}. 
\end{theorem}
Using this parameterization, we can recast the standard optimal control problem \eqref{eq:opt_ctrl} as the SLS problem
\begin{empheq}[box=\widefbox]{align}\label{eq:SLS_optimal_problem}
&\underset{\{\tf R ,\tf M,\tf N,\tf L\}}{\text{minimize}}\: J(G,\tf \Theta)\\ &\;\; := \left\| \begin{bmatrix} C_1 & D_{12} \end{bmatrix} \Thetamat \begin{bmatrix} B_1 \\ D_{21} \end{bmatrix} + D_{11} \right\| \notag\\
& \text{subject to } \;\; \eqref{eq:main_1} - \eqref{eq:main_stable} \notag \:.
\end{empheq}
In the FIR case, we use the abbreviated notation $J(g,\tf \Theta)$ for the case where $G$ is the plant $(Z,e_1,g^\tp)$.
\begin{remark}
Although we focus on the $\hinf$ optimal control problem posed in equation \eqref{eq:SLS_optimal_problem}, the results that follow carry over naturally to $\mathcal{H}_2$ (LQG) and $\mathcal{L}_1$ optimal control problems as well.
\end{remark}

\else
\input{robust-abridged}
\fi

\section{Sample Complexity Bounds}
\label{sec:optimality}
\if\MODE1

We now provide finite-data performance guarantees for a controller synthesized using the system identification and robust synthesis procedures described in the previous sections.  Prior to stating our main results, we recall the problem set up and Coarse-ID Control pipeline.

We consider the identification and control of the system $(Z,e_1,g^\tp)$, which is assumed to be FIR of order $r$.  We begin with the simplified setting that the order $r$ of the true system is known, and we use the Coarse-grained identification procedure described in Section~\ref{ssec:coarse-id} to identify an approximate system  $(Z,e_1,\tilde g^\tp)$, also of order $r$, using a series of $m$ experiments.  We then use this approximate system $(Z,e_1,\tilde g^\tp)$, as well as high-probability bounds on the estimation error $\|g - \tilde g\|_2$, in a robust SLS problem (see \eqref{eq:robust_synthesis} in Section \ref{sec:robust-sls}) to compute a controller with provable suboptimality guarantees, as formalized in the following theorem.
\footnote{Here, $\gtrsim$ hides universal constants: see Lemma \ref{lem:coarseID} for an explicit characterization.}
\begin{theorem}\label{thm:sc}
Let $\tf \Theta_0$ be the optimal solution of the SLS problem~\eqref{eq:SLS_optimal_problem} for the plant $(Z,e_1,g^\tp)$, and let $\tilde g$ be an estimate of $g$ obtained using coarse-grained identification ($\sigma^2$-variance output noise only) with $m$ experiments, where $\|\tf u_i\|_p\leq 1\:\forall i$. Let $(\Thetat_*,\alpha_*)$ be the optimal solution to the robust SLS problem  \eqref{eq:robust_synthesis}  for $(Z,e_1,\gt^\tp)$, and let $\Thetah_*$ be the response achieved on the true system $g$ by the synthesized controller  $\Kt_* = \Lt_* - \Mt_* \Rt_*^{-1} \Nt_*$. Then, if $m\gtrsim \sigma^2r\|\No\|^2\log(\eta^{-1})^{\frac{1}{2}}$, with probability at least $1-\eta$, the controller $\Kt_*$ stabilizes the true system $(Z, e_1, g^\tp)$ and has a suboptimality gap bounded by
\begin{align*}
J(g, \Thetah_*) -&  J(g,\tf \Theta_0) \\
\leq &\; 8 \sqrt{\log 2\frac{\sigma^2 r^{2/\max(p,2)}}{m}} \left(1+\sqrt{2\log\eta^{-1}}\right)\bignorm{\begin{bmatrix}1+g^\tp\No\\\Lo\end{bmatrix}} \|\Ro B_1 + \No D_{21}\|\:.
\end{align*}
\end{theorem}
\begin{coro}\label{cor:sc}
Assume that we are in the setting of Theorem~\ref{thm:sc}, and further let there be process noise with variance $\sigma_w^2$ that enters the system via the same channel as the control input (i.e., $B_1=B_2$) and measurement noise with variance $\sigma_\xi^2$. Then, Theorem~\ref{thm:sc} holds with $\sigma^2 \gets \sigma_w^2\|\tf G\|^2+\sigma_\xi^2$\:.
\end{coro}

We can further generalize these results to the setting where the order $r$ of the underlying system is not known, and that the true system is approximated by a length-$\tilde r$ FIR filter with coefficients $\tilde g$ where $\tilde r < r$. In this case, applying the triangle equality 
\begin{align*}
\|\delta\|_2 \leq \|g_{0:\tilde r-1} - \gt\|_2 + \|g_{\tilde r:r-1}\|_2
\end{align*}
gives a similar sample complexity bound, albeit one where the cost difference does not tend to zero as the number of experiments $m$ tends to infinity.
\begin{coro}\label{cor:under_sc}
Assume that we are in the setting of Theorem~\ref{thm:sc}, except let $\tilde g$ be a length-$\tilde r$ (where $\tilde r< r$) FIR estimate of $g$ obtained using the prescribed coarse-grained identification. Furthermore, assume that $\|g_{\tilde r-1:r-1}\|_2 <(4\No)^{-1}$. Then, if 
\begin{align*}
m\gtrsim \sigma^2r\left(\frac{\|\No\|}{1-4\|\No\|\|g_{\tilde r-1:r-1}\|_2}\right)^2\sqrt{\log(\eta^{-1})}\:,
\end{align*} with probability at least $1-\eta$ the controller $\Kt_* = \Lt_* - \Mt_* \Rt_*^{-1} \Nt_*$ stabilizes the true system $(Z, e_1, g^\tp)$ and has a suboptimal cost bounded by
\begin{align*}
J(g,\Thetah_*)  -&  J(g,\tf \Theta_0)   \\
\leq & \Bigg(8 \sqrt{\log 2\frac{\sigma^2 r^{2/\max(p,2)}}{m}}\left(1+\sqrt{2\log\eta^{-1}}\right)\bignorm{\begin{bmatrix}1+g^\tp\No\\\Lo\end{bmatrix}}+ 4\|g_{\tilde r-1:r-1}\|_2\Bigg)\\
&\; \times\|\Ro B_1 + \No D_{21}\|\:.
\end{align*}
\end{coro}

To prove the above results, we first derive a robust variant of the SLS framework presented in Section \ref{sec:sls_robust}, and then show how it can be used to pose a robust synthesis problem that admits suboptimality guarantees.  In particular, these guarantees characterize the degradation in performance of the synthesized controller as a function of the size of the uncertainty on the transfer function coefficients $g$.  We then combine this characterization of performance degradation with high-probability bounds on the estimation error produced by the coarse-grained identification procedure to provide an end-to-end analysis of the Coarse-ID control procedure.

\subsection{Robust SLS}
\label{sec:robust-sls}
As we only have access to approximately identified plants, we need a robust variant of Theorem \ref{thm:param}. First, we introduce a robust version of~\eqref{eq:main_2},
\begin{equation}
\Thetatmat
\begin{bmatrix} zI - A \\ -C_2 \end{bmatrix} = 
\begin{bmatrix} I + \tf\Delta_1 \\ \tf\Delta_2 \end{bmatrix}. \label{eq:robust}
\end{equation}
We call equations~\eqref{eq:main_1},~\eqref{eq:robust}, and~\eqref{eq:main_stable} the \emph{robust SLS constraints}.  We now have the ingredients needed to connect the main and robust SLS constraints. The proof is mostly algebraic and is thus deferred to the Appendix.
\begin{restatable}[Robust Equivalence]{lemma}{robustequiv}\label{lem:robust-equiv}
Consider system reponses $\Thetat=\{\Rt ,\Mt ,\Nt ,\Lt \}$ and $\Thetah=\{\Rh,\Mh,\Nh,\Lh\}$, where the latter is given by 
\begin{align}\label{eq:robust-response}
\Rh = &\; (I+\tf\Delta_1)^{-1} \Rt \\
\Mh = &\; \Mt - \tf\Delta_2(I+\tf\Delta_1)^{-1}\Rt \notag\\
\Nh = &\; (I+\tf\Delta_1)^{-1} \Nt \notag\\
\Lh = &\; \Lt - \tf\Delta_2(I+\tf\Delta_1)^{-1}\notag\Nt\:,
\end{align} where by assumption $(I+\tf\Delta_1)^{-1}$ exists and is in $\RHinf$. Let $\tf G = (A,B,C,D)$ be a given plant, and consider the following statements.
\begin{enumerate}[(i)]
    \item $\Thetat$ satisfies the robust SLS constraints for $\tf G$.
    \item $\Thetah$ satisfies the SLS constraints for $\tf G$.
\end{enumerate}
Under the assumptions, $(i)\implies (ii)$. Furthermore, let $\tf G'=(A,B,C',D)$, and let
\begin{align*}
\tf \Delta_1 = &\; -\Nt(C-C') \\
\tf \Delta_2 = &\; -\Lt(C-C') \:.
\end{align*}
Then, $(i)$ is equivalent to a third statement $(iii)$: $\Thetat$ satisfies the SLS constraints for $\tf G'$.
\end{restatable}

A chain of corollaries follow from Lemma \ref{lem:robust-equiv} that will be useful in quantifying the performance achieved on the true system of a controller designed using an approximate system model. Unless otherwise noted, let $\Thetat,\:\Thetah$ be defined as in Lemma~\ref{lem:robust-equiv}. 

\begin{coro}\label{cor:robust} Suppose that $\Thetat$ satisfies the robust SLS constraints for the system \eqref{eq:joint-realization}. Then, the controller $\Kt = \Lt - \Mt \Rt^{-1} \Nt$ stabilizes the system~\eqref{eq:joint-realization} and achieves the closed-loop system response $\Thetah$ if and only if $(I+\tf\Delta_1)^{-1} \in \RHinf$.
\end{coro}
\begin{proof}  First note that the robust SLS constraints imply that $\tf \Delta_2 \in \RHinf$.
Next, assume $(I+\tf \Delta_1)^{-1}\in \RHinf$. By Lemma~\ref{lem:robust-equiv}, $\Thetah$ satisfies the SLS constraints for~\eqref{eq:joint-realization}. Thus, by Theorem~\ref{thm:param}, $\tf{\hat K}=\Lh-\Mh\Rh^{-1}\Nh$ is stabilizing and achieves the closed-loop response $\Thetah$. Moreover, $\tf{\hat K}$ is precisely equal to $\Kt$.  

Conversely, assume $(I+\tf \Delta_1)^{-1}$ exists but is not in $\RHinf$ (if it does not exist the system response~\eqref{eq:robust-response} is obviously not well-defined). It then follows that $\Rh =  (I+\tf\Delta_1)^{-1} \Rt$ is not in $\RHinf$ as $\Rt$ is square and invertible.
\end{proof}
This immediately gives us a sufficient condition for robustness of the SLS procedure. 
\begin{coro}\label{coro:sufficient} Suppose that $\Thetat$ satisfies the robust SLS constraints for the system \eqref{eq:joint-realization}.  A sufficient condition for the controller $\Kt = \Lt - \Mt \Rt^{-1} \Nt$ to stabilize the system \eqref{eq:joint-realization} and achieve closed-loop response $\Thetah$ is that $\|\tf\Delta_1\| < 1$, for any induced norm $\|\cdot \|$.
\end{coro}
\begin{proof}
It follows from the small-gain theorem \cite{ZDGbook} that $\|\tf\Delta_1\| < 1 \implies (I+\tf\Delta_1)^{-1}\in\RHinf$, and thus Corollary~\ref{cor:robust} applies.
\end{proof}

We now specialize our results to the case where the plant $\tf P$, as defined in \eqref{eq:joint-realization}, is FIR. In this case, the modeling error arises only in the coefficient vector $g$ defining the impulse response.  To that end, we define the the estimated plant $\Gt$ with the realization $(Z, e_1, \gt^\tp)$ and note that the resulting error arises only in the $C_1$ and $C_2$ terms of the corresponding estimated plant $\Pt$, where these state-space parameters are defined as in \eqref{eq:FIR_Cs}.  To that end, we define the estimation error vector $\delta := g - \gt,$
allowing us to further specialize Corollary \ref{coro:sufficient}.
\begin{coro}\label{coro:special}
Suppose $\Thetat$ satisfies the SLS constraints for the estimated system $(Z,e_1,\gt^\tp)$. If $\|\Nt\delta^\tp\| < 1$ for any induced norm $\|\cdot\|$, then the controller $\Kt = \Lt - \Mt \Rt^{-1} \Nt$ stabilizes the true system $(Z,e_1,g^\tp)$ and achieves the closed-loop response $\Thetah$ as specified in~\eqref{eq:robust-response}. Additionally, if the induced norm $\|\cdot\|$ is either the $\mathcal{H}_\infty$ or $\mathcal{L}_1$ norm, the response $\Thetah$ simplifies to
\begin{equation}
\Thetahmat = 
\Thetatmat
+
\frac{1}{1-\delta^\tp\Nt}
\begin{bmatrix}\Nt\delta^\tp\\\Lt\delta^\tp\end{bmatrix}
\begin{bmatrix}\Rt & \Nt\end{bmatrix}\:.
\label{eq:hat-response-ez}
\end{equation}
\end{coro}
\begin{proof} By Lemma~\ref{lem:robust-equiv}, $\Thetat$ satsisfies the robust SLS constraints for $(Z,e_1,g^\tp)$. The sufficient condition then follows by applying Corollary~\ref{coro:sufficient}.Furthermore, if the induced norm $\|\cdot\|$ used in Corollary \ref{cor:robust} and Corollary \ref{coro:sufficient} is either the $\mathcal{H}_\infty$ or $\mathcal{L}_1$ norm, it follows from H{\"o}lder's inequality that $\|\Nt\delta^\tp\|<1$ implies that $|\delta^\tp\Nt|<1$ on $\mathbb{D}^c$.  Hence, we can use the Sherman-Morrison identity,
\begin{equation*}
(I - xy^\T)^{-1} = I + \frac{xy^\T}{1-y^\T x}\:, 
\end{equation*}
to simplify the closed loop response \eqref{eq:robust-response} achieved by the approximate controller $\Kt = \Lt - \Mt \Rt^{-1} \Nt$ to the expression \eqref{eq:hat-response-ez}.
\end{proof}

We now use this robust parameterization to formulate a robust SLS problem that yields a controller with stability and performance guarantees.

We will use these two facts without fanfare in the following sections.
\begin{proposition}
Let $x \in \R^n,\: \tf y \in \RHinf^n$. Then
$\|x^\tp\tf y\| \leq \|x\|_2 \|\tf y\|$.
\end{proposition}
\begin{proposition}
 $\left\| (I+\tf A)^{-1}\right\| \leq (1-\|\tf A\|)^{-1}$ for all $\|\tf A\|<1$.
\end{proposition}

Define $J(g,\tf\Theta)$ to be the performance (i.e. the objective in~\eqref{eq:SLS_optimal_problem}) of the controller $\tf K = \tf L-\tf M \tf R^{-1}\tf N$ induced by $\tf\Theta = \{\tf R, \tf M, \tf N, \tf L\}$ when placed in closed-loop with the FIR plant $G$ specified by impulse response coefficients $g$. Now, assume we design a response $\Thetat$, with corresponding controller $\Kt$,  that satisfies the SLS constraints specified by the estimate system $\gt$. We saw in the previous section that under suitable conditions, the response on the true system $g$ is given by $\Thetah$, as specified in Corollary \ref{coro:special}. By the triangle inequality, Corollary~\ref{coro:special}, and our parametric assumption~\eqref{eq:FIR_Cs}-\eqref{eq:FIR_D12}, we can then bound the difference between expectation $J(\gt,\Thetat)$ and reality $J(g,\Thetah)$ as follows:
\begin{align*}
J(g,&\;\Thetah) 
=  \left\| \begin{bmatrix} C_1 & D_{12} \end{bmatrix} \Thetahmat \begin{bmatrix} B_1 \\ D_{21} \end{bmatrix} + D_{11} \right\|\\ 
= &\; \left\| \begin{bmatrix} g^\tp&0\\0&1\end{bmatrix}\left(\Thetatmat +
\begin{bmatrix}\frac{\Nt\delta^\tp}{1-\delta^\tp\Nt}\\\frac{\Lt\delta^\tp}{1-\delta^\tp\Nt}\end{bmatrix}
\begin{bmatrix}\Rt & \Nt\end{bmatrix}\right)  \begin{bmatrix} B_1 \\ D_{21} \end{bmatrix} + D_{11} \right\|\\
= &\; 
\left\|\left(\left(\begin{bmatrix} \gt^\tp&0\\0&1\end{bmatrix}+\begin{bmatrix} \delta^\tp&0\\0&0\end{bmatrix} \right)\Thetatmat + \begin{bmatrix} g^\tp&0\\0&1\end{bmatrix}
\begin{bmatrix}\frac{\Nt\delta^\tp}{1-\delta^\tp\Nt}\\\frac{\Lt\delta^\tp}{1-\delta^\tp\Nt}\end{bmatrix}
\begin{bmatrix}\Rt & \Nt\end{bmatrix}\right)  \begin{bmatrix} B_1 \\ D_{21} \end{bmatrix} + D_{11} \right\|
\\
\leq &\; J(\gt,\Thetat) + \left\| \left(
\begin{bmatrix}
\delta^\tp \\0
\end{bmatrix}
+ \begin{bmatrix} g^\tp&0\\0&1\end{bmatrix}\begin{bmatrix}\frac{\Nt\delta^\tp}{1-\delta^\tp\Nt}\\\frac{\Lt\delta^\tp}{1-\delta^\tp\Nt}\end{bmatrix}
\right)\begin{bmatrix}\Rt & \Nt\end{bmatrix}\begin{bmatrix} B_1 \\ D_{21} \end{bmatrix} \right\| \\
\leq &\; J(\gt,\Thetat) + \frac{1}{1-\|\delta^\tp\Nt\|} \left\|
\begin{bmatrix}\delta^\tp + \gt^\tp\Nt\delta^\tp\\\Lt\delta^\tp\end{bmatrix}
\right\| \|\Rt B_1 + \Nt D_{21} \|\\
\leq &\; J(\gt,\Thetat) +   \frac{\|\delta\|}{1-\|\delta^\tp\Nt\|} \left\|
\begin{bmatrix} 1+ \gt^\tp\Nt\\\Lt\end{bmatrix}
\right\| \|\Rt B_1 + \Nt D_{21} \|\:,
\end{align*}
where we assume $\|\Nt\delta^\tp\|<1$ for the bound to be valid.

For any estimated response $\gt$ satisfying $\|\delta\|\leq \epsilon$, it then follows that
\begin{equation}
J(g,\Thetah) \leq J(\gt,\Thetat) + \epsilon\alpha\|\Rt B_1 + \Nt D_{21} \|\label{eq:robust-objective}
\end{equation}
for any $\alpha$ satisfying
\begin{equation}
\epsilon\alpha\|\Nt\| + \left\|
\begin{bmatrix} 1+ \gt^\tp\Nt\\\Lt\end{bmatrix}
\right\| \leq \alpha\:, \label{eq:robust-constraint}
\end{equation}
noting that $\epsilon\|\Nt\|<1$, which implies $\|\Nt\delta^\tp\|<1$, is equivalent to $\alpha > 0$. We denote the right-hand side of this bound as 
\begin{equation*}
Q(\gt,\Thetat,\alpha) := J(\gt,\Thetat)+  \epsilon\alpha\|\Rt B_1 + \Nt D_{21} \| \:.
\end{equation*}
The bound~\eqref{eq:robust-objective} then suggests the following robust controller synthesis procedure, which balances between solving for the optimal controller for the approximate system $\gt$ and controlling a perturbative term. We call this problem the \emph{robust SLS problem for $\gt$}.
\begin{empheq}[box=\widefbox]{align}
\underset{\substack{\{\Rt,\Mt,\Nt,\Lt\}\\ \alpha>0}}{\minimize} &\;\; Q(\gt,\Thetat,\alpha) \label{eq:robust_synthesis}\\
\st &\;\; \text{$\Thetat$ satisfies SLS constraints}\notag\\
&\;\text{\eqref{eq:main_1} - \eqref{eq:main_stable} for $(Z,e_1,\gt^\tp)$} \notag\\
&\; \epsilon\alpha\|\Nt\| + \left\|
\begin{bmatrix} 1+ \gt^\tp\Nt\\\Lt\end{bmatrix}
\right\| \leq \alpha\:.\notag
\end{empheq}
Although this problem is not jointly convex in $\alpha$ and the system responses $\Thetat$, one can use a golden section search on $\alpha$ in practice. Moreover, the sum of norms can be split into two norm constraints using an epigraph formulation (see~\cite{boyd2004convex}, Ch. 3).

\subsection{Sub-optimality guarantees for robust SLS}

We now show a bound on the change in the optimal control cost when the controller is synthesized using the robust SLS problem \eqref{eq:robust_synthesis}.

\begin{proposition}\label{prop:optimality}
Let $\tf \Theta_0$ and $(\Thetat_*,\alpha_*)$, as well as $\Thetah_*$, be defined as in Theorem~\eqref{thm:sc}, and let $\|\delta\| \leq \epsilon$. If $\epsilon < (2\|\No\|)^{-1}$, we  have that
\begin{align}
J(g,&\Thetah_*) - J(g,\tf \Theta_0) \leq  \label{eq:final-bound} \\
&\;\frac{2\epsilon}{1-2\epsilon\|\No\|}\bignorm{\begin{bmatrix}1+g^\tp\No\\\Lo\end{bmatrix}} \|\Ro B_1 + \No D_{21}\|\:. \notag
\end{align}
\end{proposition}
To prove this proposition, we require a technical lemma that ensures that the true controller $\Ko$ stabilizes the estimate system specified by the FIR coefficients $\gt$, i.e. that the optimal system response $\tf \Theta_0$ can be used to construct a feasible solution to the approximate SLS synthesis problem \eqref{eq:robust_synthesis}.
\begin{lemma}\label{lem:feasibility}
Let $\tf \Theta_0$ and its induced controller $\Ko$ be as defined in Theorem~\ref{thm:sc}, and let $||\delta||\leq \epsilon$, with $\epsilon < (2\|\No\|)^{-1}$. Then
\begin{equation}
\alpha_0 := \frac{1}{1-2\epsilon\|\No\|}\left\|\begin{bmatrix} 1+ g^\tp \No\\ \Lo\end{bmatrix}\right\|\:
\end{equation}
is strictly positive, and the controller $\Ko$ is stabilizing for the estimate system specified by $(Z,e_1,\gt^\tp)$ and achieves the system response $\Thetah_0$ defined by
\begin{equation}
\begin{array}{rcl}
\Rh_0 &=& (I+\No\delta^\tp)^{-1}\Ro \\
\Mh_0 &=& \Mo - \Lo\delta^\tp(I+\No\delta^\tp)^{-1}\Ro \\
\Nh_0 &=& (I+\No\delta^\tp)^{-1}\No \\
\Lh_0 &=& \Lo - \Lo\delta^\tp(I+\No\delta^\tp)^{-1}\No.
\end{array}
\label{eq:hats}
\end{equation}
Furthermore, $(\Thetah_0, \alpha_0)$ are feasible solutions to the approximate SLS synthesis problem~\eqref{eq:robust_synthesis}.
\end{lemma}
\begin{proof}
Both of these points are conditional on $(I+\No\delta^\tp)^{-1}$ existing in $\RHinf$:
\begin{itemize}
    \item By assumption, $\tf\Theta_0$ satisfies the SLS constraints for $(Z,e_1,g^\tp)$; by Lemma~\ref{lem:robust-equiv}, it equivalently satisfies the robust SLS constraints for $(Z,e_1,\gt^\tp)$, where $\tf\Delta_1 = \No\delta^\tp$ and $\tf\Delta_2 = L_0 \delta^\tp$ (note the switched roles of $g$ and $\gt$). By Corollary~\ref{cor:robust}, $\Ko$ then stabilizes $(Z,e_1,\gt^\tp)$ and achieves the system response $\Thetah_0$.
    \item By Lemma~\ref{lem:robust-equiv}, $\Thetah_0$ satisfies the SLS constraints for $(Z,e_1,\gt^\tp)$, and is thus part of a feasible point for the approximate SLS synthesis problem~\eqref{eq:robust_synthesis}. Now, we need to check that the corresponding $\alpha_0$ is also part of a feasible solution. Toward that end, by the Sherman-Morrison identity, we see that 
    \begin{equation*}
     \|\epsilon\Nh_0\| = \; \| \epsilon (I-\frac{\No\delta^\tp}{1+\delta^\tp\No})\No\|
    \leq   \epsilon\|\No\| + \frac{(\epsilon\|\No\|)^2}{1-\epsilon\|\No\|} = \frac{\epsilon \|\No\|}{1-\epsilon\|\No\|}.
    \end{equation*}
    Furthermore,
    \begin{equation*}
    \left\|\begin{bmatrix} 1+ \gt^\tp \Nh_0\\ \Lh_0\end{bmatrix} \right\| =\; \left\|\begin{bmatrix} \frac{1+g^\tp\No}{1+\delta^\tp\No}\\ \frac{\Lo}{1+\delta^\tp\No}\end{bmatrix} \right\| \leq  \frac{1}{1-\epsilon\|\No\|}\left\|\begin{bmatrix} 1+ g^\tp \No\\ \Lo\end{bmatrix}\right\|.
    \end{equation*}
    Therefore,
    \begin{align*}
    \frac{1}{1-\|\epsilon\Nh_0\|}\left\|\begin{bmatrix} 1+ \gt^\tp \Nh_0\\ \Lh_0\end{bmatrix} \right\| \leq \alpha_0\:,
    \end{align*}
    the final feasibility condition of~\eqref{eq:robust_synthesis}.
\end{itemize}
It therefore remains to verify that $(I+\No\delta^\tp)^{-1}$ exists and is in $\RHinf$. As we have seen, a sufficient condition is $\|\No\delta^\tp\| < 1$, and this condition is implied by the assumption that $\alpha_0 > 0$.
\end{proof}
\begin{proof}[Proof of Proposition \ref{prop:optimality}]
We immediately invoke Lemma~\ref{lem:feasibility} by noting that our assumption on $\epsilon$ ensures $\alpha_0>0$, and we are assured that $(\Thetah_0,\alpha_0)$ is a feasible point for the approximate SLS synthesis problem~\eqref{eq:robust_synthesis}. From inequality~\eqref{eq:robust-objective}, we then have that
\begin{align}
J(g,\Thetah_*) \leq \; Q(\gt,\Thetat_*,\alpha_*)\notag
\leq &\; Q(\gt,\Thetah_0,\alpha_0) \notag\\
= &\; J(\gt,\Thetah_0) + \epsilon \alpha_0 \|\Rh_0 B_1 + \Nh_0 D_{21} \| \notag\\
\leq&\; J(\gt,\Thetah_0) + \epsilon \alpha_0\frac{\|\Ro B_1 + \No D_{21}\|}{1-\epsilon\|\No\|}\;, \label{eq:bound1}
\end{align}
where the second inequality follows from the optimality of $(\Thetat_*, \alpha_*)$, and the final inequality from the definitions of $\Rh_0$ and $\Nh_0$.
Now, we repeat the argument used to derive~\eqref{eq:robust-objective} with expectation and reality reversed: this time we assume our design expectation was $J(g,\tf\Theta_0)$ but our reality is $J(\gt,\Thetah_0)$. This is a valid analogy as $\tf\Theta_0$ satisfies the SLS equations for $(Z,e_1,g^\tp)$.  With the true and estimated parameters reversed, we can thus bound $J(\gt,\Thetah_0)$ by
\begin{equation}
J(\gt,\Thetah_0) \leq\; J(g,\tf \Theta_0) + \frac{\epsilon}{1-\epsilon\|\No\|}\bignorm{\begin{bmatrix}1+g^\tp\No\\\Lo\end{bmatrix}}
\|\Ro B_1 + \No D_{21}\| \:. \label{eq:bound2}
\end{equation}
Finally, combining bounds \eqref{eq:bound1} and \eqref{eq:bound2} and plugging in $\alpha_0$ gives
\begin{align*}
J(g,\Thetah^*) - J(g,\tf \Theta_0) \leq &\; \left(\alpha_0 + \bignorm{\begin{bmatrix}1+g^\tp\No\\\Lo\end{bmatrix}}\right)\frac{\epsilon\|\Ro B_1 + \No D_{21}\|}{1-\epsilon\|\No\|} \\
= &\; \frac{2\epsilon}{1-2\epsilon\|\No\|}\bignorm{\begin{bmatrix}1+g^\tp\No\\\Lo\end{bmatrix}} \|\Ro B_1 + \No D_{21}\|\:.
\end{align*}
\end{proof}
\subsection{Coarse-grained ID and the proof of Theorem~\ref{thm:sc}}
First, to prove the sample complexity of synthesizing a stabilizing controller based on an approximate system, we require an intermediary lemma 
on how well coarse-grained identification can identify the true system. The proof of the lemma (and the related change necessary for Corollary~\ref{cor:sc}) is deferred to the Appendix.
\begin{restatable}{lemma}{coarseID}\label{lem:coarseID}
Assume we estimate the system $g$ by a length-$r$ FIR system $\tilde g$ using coarse-grained identification (output noise only) on $m$ experiments, where the inputs $\tf u_i$ are constrained to lie in a unit $\ell_p^T$ ball. Then, with probability at least $1-\eta$, 
\begin{equation*}
||\delta||_2 \leq 2 \sqrt{\log 2\frac{\sigma^2r^{2/\max(p,2)}}{m}}\left(1+\sqrt{2\log\eta^{-1}}\right)\:.
\end{equation*}
\end{restatable}
Taking $m$ large enough such that $\|\delta\|_2 = \epsilon < (4\|\No\|)^{-1}$ (implied by taking $m\gtrsim \sigma^2r\|\No\|^2\sqrt{\log(\eta^{-1})}$), we have 
\begin{align*}
\frac{2\epsilon}{1-2\epsilon\|\No\|} \leq &\; 8 \sqrt{\log 2\frac{\sigma^2r^{2/\max(p,2)}}{m}}\left(1+\sqrt{2\log\eta^{-1}}\right)\:.
\end{align*}
Finally, we show that $\Kt_*$ is stabilizing for the true system $(Z,e_1,g^\tp)$. Since $(\Thetat_*,\alpha_*)$ is optimal for the approximate SLS synthesis problem for $\gt$, it is feasible, and thus $\alpha_*>0$ allows us to invoke Corollary~\ref{coro:special}, as we have that $\|\Nt_*\delta^\tp\|< 1\:$.

\else
\input{optimality-abridged}
\fi

\section{Experiments}
\label{sec:experiments}
\if\MODE1

The robust SLS procedure analyzed in the previous section requires solving an infinite-dimensional optimization problem as the responses $\{\Rt,\Nt,\Mt,\Lt\}$ are not required to be FIR. However, as an approximation, we limit them to be FIR responses of a prescribed length $T$.  By making this restriction, the resulting optimization problem is then finite-dimensional and admits an efficient solution using off-the-shelf convex optimization solvers\footnote{Code for these computations can be found at \texttt{https://github.com/rjboczar/OF-end-to-end-CDC}}.

Figure \ref{fig:approx} shows a quantification of this approximation. In this experiment, for each $r$, we chose random FIR plants with impulse response coefficients uniformly distributed in $[-1,1]^r$. We then computed the smallest $T(r)$ such that the robust performance returned by the SLS program was within 2\% relative error of the performance calculated by MATLAB's \texttt{hinfsyn} with relative tolerance $10^{-8}$. Figure \ref{fig:approx} also shows this calculation when each plant was normalized to have unit $\hinf$-norm.
\begin{figure}[t]
\centering
\includegraphics[width=.8\linewidth]{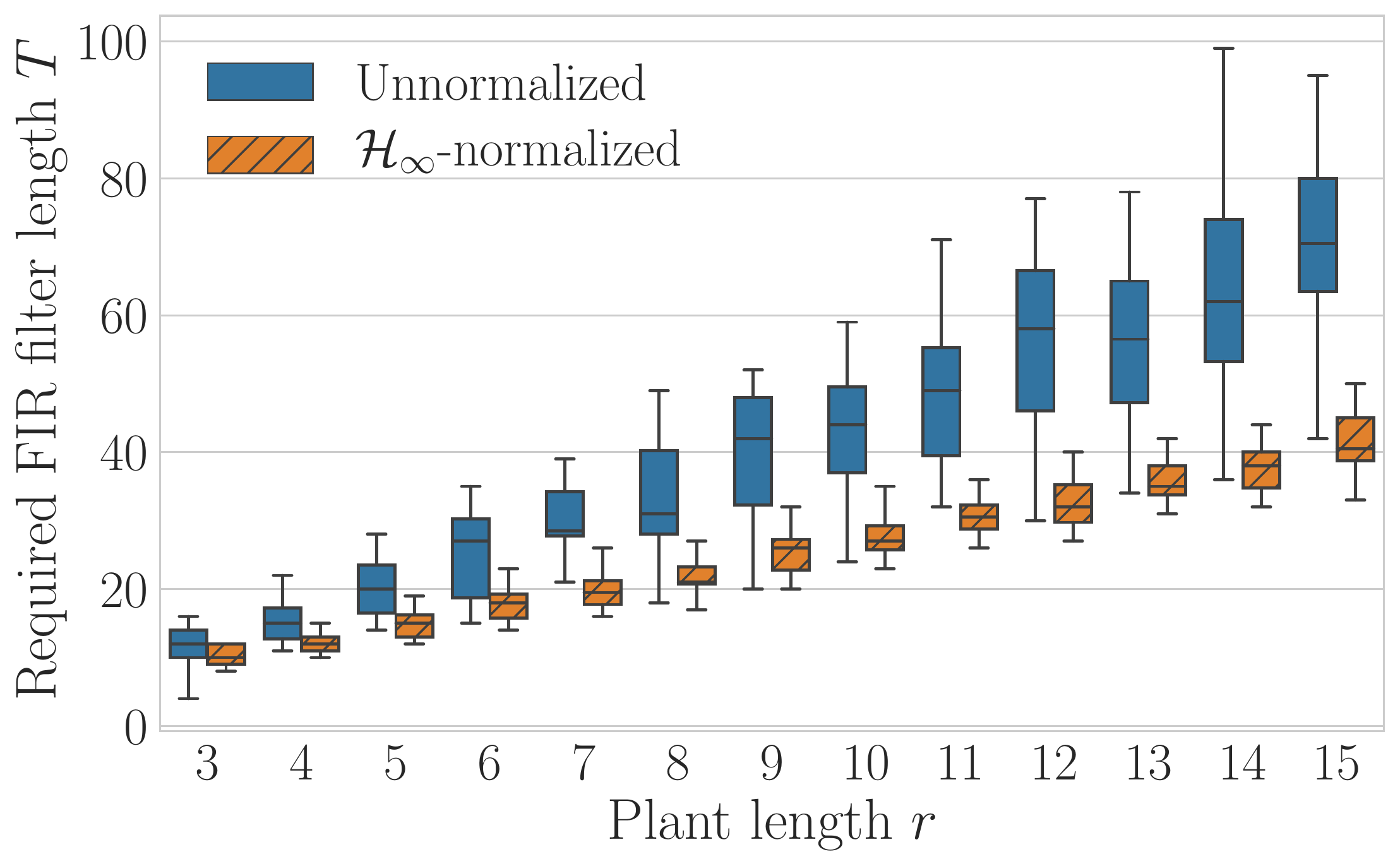}
\caption{FIR approximation length $T$ required to achieve small relative error in the robust performance objective. For each plant length $r$, the boxes denote the middle quartiles for $10$ random plants, and the whiskers show the extents of the data.}
\label{fig:approx}
\end{figure}

\subsection{Optimization Model}
Let $\vec(\tf F) := \left[\tf F_0^\tp \:\ldots \tf\: \tf F_{n-1}^\tp\right]^\tp$, $\rvec(\tf F) := \left[\tf F_0 \: \ldots \tf\: \tf F_{n-1}\right]$, and $\tf I$ be the (static) identity transfer function. Furthermore, appealing to the SDP characterization of $\hinf$-bounded FIR systems (\cite{dumitrescu2017positive} Thm. 5.8), define 
\begin{align*}
H(Q,\tf F,\gamma) := &\; \begin{bmatrix}
Q & \vec(\tf F) \\ \vec(\tf F)^\tp & \gamma I 
\end{bmatrix}\:.
\end{align*}
Then, under the approximate assumption that $\Rt, \Nt, \Mt, \Lt$ are FIR of length $T$, and using the notation $Q_{[j,k]}$ for the $(j,k)$-th block of $Q$, we can write the full optimization problem of solving~\eqref{eq:robust_synthesis} for a fixed $\alpha$:
\begin{align*}
\underset{\substack{\{\Rt,\Mt,\Nt,\Lt\},\\ t_1, 
\: t_2,\:t_3,\:t_4}}{\minimize}&\; t_1 + t_2 \\
\st \;&
\begin{bmatrix}\rvec(\Rt) & 0\end{bmatrix} -\begin{bmatrix}0 & A\rvec(\Rt)\end{bmatrix} =  \begin{bmatrix}0 & B_2\rvec(\Mt) + \rvec(\tf I)\end{bmatrix}\\
&\begin{bmatrix}\rvec(\Nt) & 0\end{bmatrix} -\begin{bmatrix}0 & A\rvec(\Nt)\end{bmatrix} = \begin{bmatrix}0 & B_2\rvec(\Lt)\end{bmatrix}
\\
& \begin{bmatrix}\vec(\Rt) \\ 0\end{bmatrix} -\begin{bmatrix}0 \\ \vec(\Rt)A\end{bmatrix} =\begin{bmatrix}0 \\ \vec(\Nt)C_2+\vec(\tf I)\end{bmatrix} \\
&\begin{bmatrix}\vec(\Mt) \\ 0\end{bmatrix} -\begin{bmatrix}0 \\ \vec(\Mt)A\end{bmatrix} = \begin{bmatrix}0 \\ \vec(\Lt)C_2\end{bmatrix} \\
&\substack{1\leq i \leq 4 \\0\leq k\leq n}
\begin{cases}
H(Q^{(i)},\tf F^{(i)},t_i) \succeq 0\\
\sum_{j=1}^{T-k}Q^{(i)}_{[j+k,j]} = \delta_kt_i I
\end{cases}\\
& \tf F^{(1)} = \begin{bmatrix} \gt^\tp&0\\0&1\end{bmatrix}\begin{bmatrix}\Rt & \Nt \\ \Mt & \Lt\end{bmatrix}\begin{bmatrix}B_1 \\ D_{21}\end{bmatrix} + D_{11} \\
& \tf F^{(2)} = \epsilon \alpha \left(\Rt B_1 + \Nt D_{21}\right) \\
& \tf F^{(3)} = \epsilon \alpha\Nt, \quad\tf F^{(4)} = \begin{bmatrix} 1+ \gt^\tp\Nt\\\Lt\end{bmatrix}\\
& t_3 + t_4 \leq \alpha\:.
\end{align*}

\subsection{Computational Results} Instead of using Lemma~\ref{lem:coarseID} directly, we use a simulation-based technique\footnote{The technique involves inverting the Chernoff bound to generate random variable tail bounds that hold with high probability with respect to the simulated instances. See the Appendix of \cite{tu17} for details.} (based on looking at the empirical histogram) to achieve tighter probabilistic tail bounds on $\|\delta\|_2$.

In what follows, we consider the following quantities: 
\begin{enumerate}[(i)]
    \item $J_{\text{nominal}}$: the cost achieved on the true system when the controller was designed using \texttt{hinfsyn} with the approximate system
    \item $\hat J$: the cost achieved on the true system when the controller was designed using the approximate SLS synthesis procedure
    \item $\delta J = \frac{J_{\text{nominal}} - \hat J}{J_{\text{nominal}}}$: the relative improvement of the approximate SLS synthesis procedure
    \item $\Delta J,\;\widehat{\Delta J}$: the theoretical sub-optimality bound~\eqref{eq:final-bound} and the actual sub-optimality gap, respectively.
\end{enumerate}

Figure~\ref{fig:nominal} shows $\delta J$ for random instances of $\hinf$-normalized plants  of different lengths $r$, swept across number of experiments $m$. 
\begin{figure}[t]
\centering
\includegraphics[width=.75\linewidth]{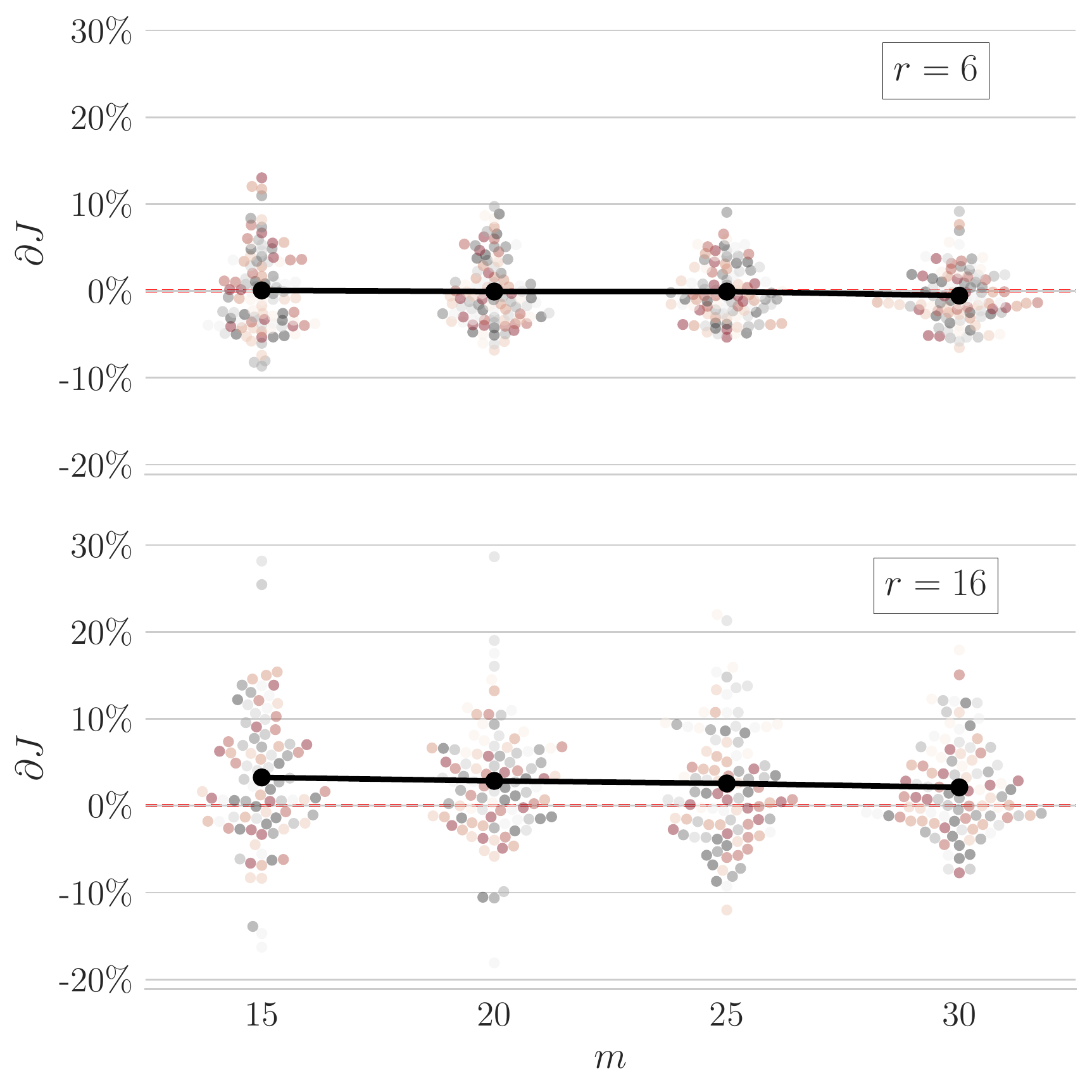}
\caption{``Swarm'' plot of relative improvement $\delta J$ from using the approximate SLS procedure across multiple random instances of plants and output noise ($8\times 10,\: \sigma=0.1$), for number of experiments $m\in\{15,20,25,30\}$. Each color dot represents a different plant.}
\label{fig:nominal}
\end{figure}

While it is difficult to draw precise quantitative conclusions from a suite of random plants, for the longer plants ($r=16$) the approximate SLS procedure does perform better on average. We hypothesize that the performance depends on an effective signal-to-noise ratio (SNR): $\frac{m}{\sigma^2 r}$. At low SNR, there may not be enough data for the approximate SLS procedure to be valid (i.e. for $\epsilon < (2\|\No\|)^{-1}$). At high SNR, $\gt$ is very close to $g$ and thus the approximate SLS procedure may be too conservative. Thus, large improvements may be hard to come by, which is seen as $m$ increases or $r$ decreases in Figure~\ref{fig:nominal}. Therefore, in between these cases may be where the procedure is most effective---the $r=16$ case may lie in this regime.
\begin{figure}[t]
\centering
\includegraphics[width=.6\linewidth]{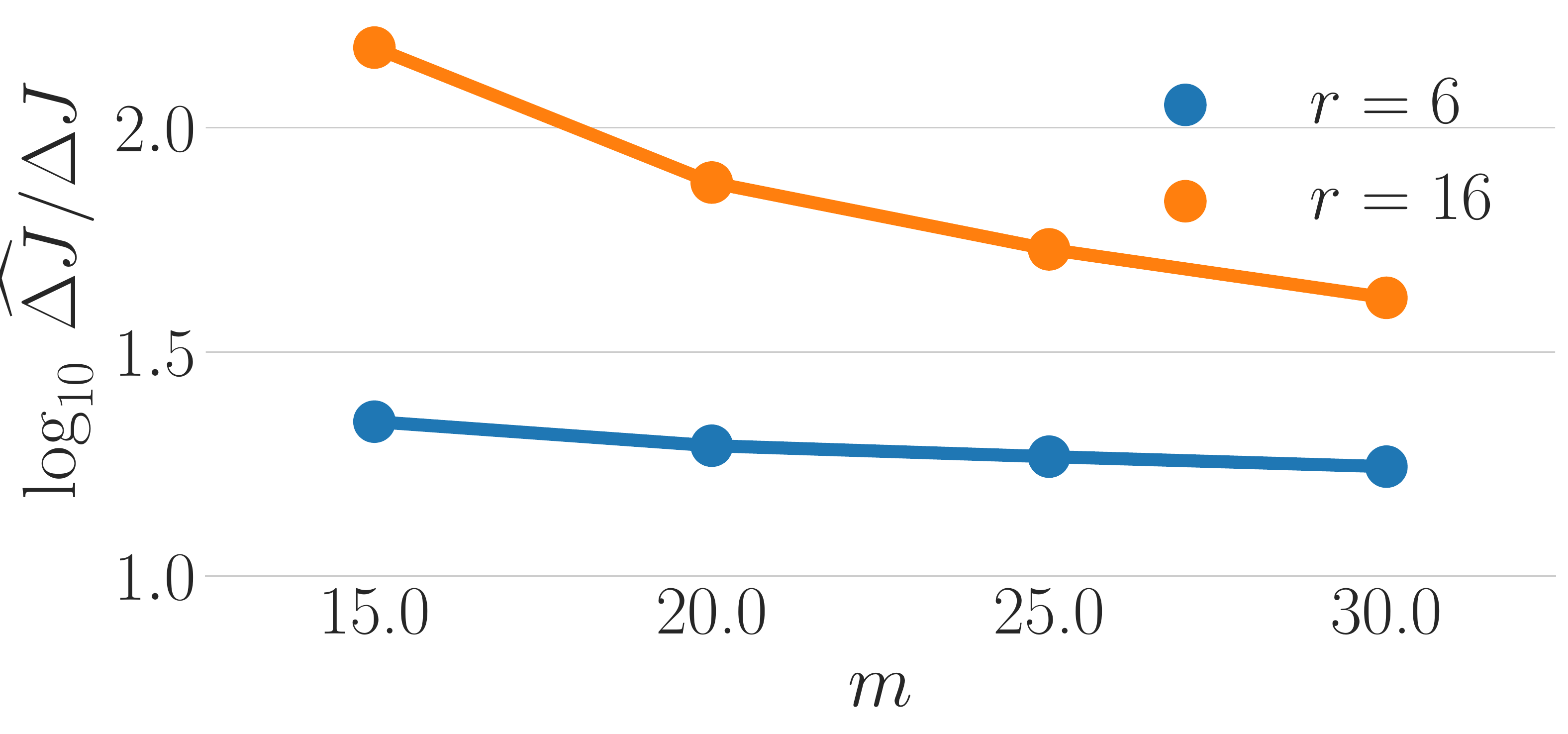}
\caption{Comparison on upper bound $\Delta J$ and actual suboptimality gap $\widehat{\Delta J}$.}
\label{fig:bound}
\end{figure}

Using the data generated for Figure~\ref{fig:nominal}, Figure~\ref{fig:bound} shows the looseness of the end-to-end bound in Proposition~\ref{prop:optimality}. This looseness is somewhat unsurprising, as both the formulation of the approximate SLS problem and the proof of Proposition~\ref{prop:optimality} feature multiple uses of the triangle inequality.
\else
\input{experiments-abridged}
\fi

\section{Conclusion}
\label{sec:conclusion}

In this work, we provide a computational tool for optimal output feedback control for the Coarse-ID setting, as well as a non-asymptotic analysis of its performance. Future work involves relaxing assumptions to allow IIR or unstable plants, the latter of which may require significant modifications to the Coarse-ID analysis. 


\begin{small}
\bibliographystyle{IEEEtran}
\bibliography{of}
\end{small}

\if\MODE1
\clearpage
\begin{appendices}
\section{Proof of Lemma~\ref{lem:robust-equiv}}
We restate the lemma for convenience.
\robustequiv*
\begin{proof}
\begin{itemize}
    \item (i)$\iff$(iii): The SLS constraints for $\tf G'$ and the robust SLS constraints for $\tf G$ are identical, by the definitions of $\tf \Delta_1$ and $\tf \Delta_2$. 
    \item (i)$\implies$(ii): 
    Satisfaction of~\eqref{eq:robust} under $(\Thetat, \tf G)$ implies satisfaction of~\eqref{eq:main_2} under $(\Thetah, \tf G)$ as they are related by a linear transformation defined by
    \begin{align*}
    V = &\; \begin{bmatrix}(I+\tf \Delta_1)^{-1} & 0\\
    -\tf\Delta_2(I+\tf\Delta_1)^{-1}& I\end{bmatrix}\:.
    \end{align*}
    We readily see that
    \begin{align*}
    V \begin{bmatrix} \Rt & \Nt \\ \Mt & \Lt \end{bmatrix}  = &\; \begin{bmatrix} \Rh & \Nh \\ \Mh & \Lh \end{bmatrix}\quad\text{and}\quad
    V \begin{bmatrix}I+\tf\Delta_1\\\tf\Delta_2\end{bmatrix} =
    \begin{bmatrix}I\\0\end{bmatrix} \:.
    \end{align*}     

    Next, satisfaction of~\eqref{eq:main_1} under $(\Thetah, \tf G)$ is implied by satisfaction of~\eqref{eq:main_1} and~$\eqref{eq:robust}$ under $(\Thetat, \tf G)$. To see this, note that multiplying~$\eqref{eq:robust}$ on the right by $[zI-A;\:-C_2]$ gives $B_2\tf\Delta_2 = (zI-A)\tf\Delta_1$. One can then verify that
    \begin{align*}
    \begin{bmatrix}zI-A&-B_2\end{bmatrix}
    \begin{bmatrix} \Rh & \Nh \\ \Mh & \Lh \end{bmatrix} = &\; \begin{bmatrix}I&0\end{bmatrix}\:.
    \end{align*}

    Finally, consider~\eqref{eq:main_stable}. The second equation of~\eqref{eq:robust}, i.e.
    \begin{align*}
    \Mt(zI-A) - \Lt C_2 = \tf\Delta_2\:,
    \end{align*}
    shows that $\tf\Delta_2\in\RHinf$. Combined with our assumption that $(I+\tf\Delta_1)^{-1}\in\RHinf$, by the definition of $\Thetah$ we may assert that $z\Rh,z\Mh,z\Nh,\Lh\in\RHinf$.
\end{itemize}
\end{proof}
\section{Coarse-grained ID Results}
We restate the lemma for convenience.
\coarseID*
\subsection{Setup}
We repeat the coarse-grained identification setup from Tu et al.~\cite{tu17}.
We are given query access to
$G$ via the form
\begin{align*}
    \tf y_i = \tf G\tf u_i + \xi \:, \:\: \xi\: \iid\:  \mathcal{N}(0, \sigma^2 I_{T}) \:,
\end{align*}
where $u \in \mathcal{U}$ for some fixed set $\mathcal{U}$.  Now, fix a set of $m$ inputs $u_1, ..., u_m \in \mathcal{U}$. Given the resulting outputs $\{\tf y_i\}_{i=1}^{m}$, we can estimate the first $T$ coefficients $\{g_k\}_{k=0}^{T-1}$ of $G(z) = \sum_{k=0}^{\infty} g_k z^{-k}$ via ordinary least-squares. Calling the vector $Y := (\tf y_1, \ldots, \tf y_m) \in \R^{Tm}$, the least-squares estimator
$\gt_{0:T-1}$ is given
by
\begin{align*}
  \gt_{0:T-1} := &\;\begin{bmatrix}
    \gt_0 &
    \gt_1 &
    \cdots &
    \gt_{T-1} &
  \end{bmatrix}^\tp
  = (Z^\tp Z)^{-1} Z^\tp Y \:,\\
\text{where}\,  Z := &\; \begin{bmatrix} \Toep(\tf u_1)^\tp & \cdots & \Toep(\tf u_m)^\tp \end{bmatrix}^\tp \in \R^{Tm \times T} \:.
\end{align*}

Explicitly, for a vector $u \in \R^T$, $\Toep(u)$ is the $T \times T$ lower-triangular Toeplitz matrix where the first column is equal to $u$. We assume the matrix $Z^\tp Z$ is invertible, which is verified for the inputs we use in \cite{tu17}. Assuming $r\leq T$, by the Gaussian noise assumption, the error vector $\delta := g_{0:r-1} - \tilde g_{0:r-1}$ will also be Gaussian with a prescribed covariance. Since the covariance matrix  will play a critical role in our analysis to follow, we introduce the notation
\begin{align*}
  \Sigma(u) := \sum_{k=1}^{m} \Toep(\tf u_k)^\tp \Toep(\tf u_k) \:,
\end{align*}
where $m$ will be clear from context.
We will also use the shorthand notation $M_{[r]}$
to refer to upper left $r \times r$ block of $M$.

\subsection{Proof of Lemma~\ref{lem:coarseID}}
We readily see that $\delta \sim \mathcal{N}(0,\sigma^2 S)$ where $S = (\Sigma(u)^{-1})_{[r]}$.  Then, noting that $||\delta||_2$ is a $\sigma\|S^{1/2}\|$-Lipschitz function of i.i.d. standard Gaussian random variables, from concentration of measure (see~\cite{boucheron2013concentration}, Theorem 5.6) we have that
\begin{align*}
P(||\delta||_2 \geq \E||\delta||_2 + t) \leq e^{-t^2/(2\sigma^2\|S\|)}\:,
\end{align*}
for all $t\geq 0$. Furthermore, by Jensen's inequality,
\begin{align*}
\E||\delta||_2 \leq \sqrt{\sum_{i=1}^r\E \delta_i^2} = \sigma\sqrt{\Tr(S)}\:.
\end{align*}{}
This then gives 
\begin{equation*}
||\delta||_2 \leq \sigma\left(\|S\|^{1/2}\sqrt{2\log\eta^{-1}} + \sqrt{\Tr(S)}\right)
\end{equation*}
with probability at least $1-\eta$. To probabilistically guarantee small error, we would then like to minimize the right hand side over input signals. When $\mathcal U$ is a unit $\ell_p$-ball in $\R^T$,~\cite{tu17} provides relevant bounds on $S$.
\begin{lemma}[c.f. Section 3~\cite{tu17}] If $\mathcal U$ is a unit $\ell_p$-ball, then 
\begin{align*}
\min_{\tf u_i \in \mathcal{U}} \Tr(S) &\; \begin{cases}
= \frac{r}{m}, & p\in[1,2]\quad (\tf u_i^* = e_1) \\
\leq \frac{4\log 2}{m}r^{2/p}, & p\in (2,\infty] \:.
\end{cases}
\end{align*}
\end{lemma}
Thus, since $||S||\leq\Tr(S)$, in the most general case we have that 
\begin{equation*}
||\delta||_2 \leq 2\sigma \sqrt{\log 2\frac{r^{2/\max(p,2)}}{m}}\left(1+\sqrt{2\log\eta^{-1}}\right)
\end{equation*}
with probability at least $1-\eta$\:. If $p\in[1,2]$, $||S||$ can be computed to be $\frac{1}{m}$, and we have
\begin{equation*}
||\delta||_2 \leq 2\sigma \sqrt{\frac{r}{m}}\left(1+\sqrt{\frac{2\log\eta^{-1}}{r}}\right)
\end{equation*}
Finally, with respect to Corollary~\ref{cor:sc}, as noted in Section 3 of~\cite{tu17} we have that
$\delta\sim \mathcal{N}(0,\Lambda)$, where
\begin{align*}
\Lambda = &\;(Z^\tp Z)^{-1}Z^\tp(\sigma_w^2 \Toep(g)\Toep(g)^\tp+\sigma_\xi^2 I)Z(Z^\tp Z)^{-1}\\
\preceq &\; (\sigma_w^2\|\tf G\|_{\hinf}^2+\sigma_\xi^2)(Z^\tp Z)^{-1},
\end{align*}
where the inequality comes from the stability of $\tf G$.
\end{appendices}

\else
\input{appendix-abridged}
\fi

\end{document}